\documentclass[12pt]{amsart}
\usepackage{amsmath,amsthm,amssymb,amsfonts,eucal,fullpage,latexsym,mathrsfs,stmaryrd}

\makeatletter
\@namedef{subjclassname@2020}{%
  \textup{2020} Mathematics Subject Classification}
\makeatother

\usepackage{enumerate}
\usepackage{amssymb}
\usepackage{amsthm}
\usepackage{amsmath}
\usepackage{color}
\usepackage{amsmath}
\usepackage{amssymb}
\usepackage{amsfonts,mathrsfs,fullpage}
\usepackage{lipsum}
\usepackage{array}
\usepackage{lineno}
\theoremstyle{plain}

\numberwithin{equation}{section}
\theoremstyle{plain}
\newtheorem{Proposition}[equation]{Proposition}
\newtheorem{Corollary}[equation]{Corollary}
\newtheorem*{Corollary*}{Corollary}
\newtheorem{Theorem}[equation]{Theorem}
\newtheorem*{Theorem*}{Theorem}

\theoremstyle{definition}
\newtheorem{Definition}[equation]{Definition}

\newtheorem{Example}[equation]{Example}

\def\phi{\varphi}

\def\D{\mathbb{D}}

\def\phi{\varphi}
\def\M{\mathscr{M}}

\renewcommand{\leq}{\leqslant}
\renewcommand{\geq}{\geqslant}
\renewcommand{\subset}{\subseteq}


\begin{document}



\title{On the geometry of the multiplier space of $\ell_A^p$}
\author[Cheng]{Raymond Cheng}
\address{Department of Mathematics and Statistics, Old Dominion University, Norfolk, VA 23529, USA. } \email{rcheng@odu.edu}
\author[Felder]{Christopher Felder}
\address{Department of Mathematics and Statistics, Washington University in St.\ Louis, St.\ Louis, MO 63130, USA. } \email{cfelder@wustl.edu}

\date{\today}

\begin{abstract}
For $p \in (1,\infty)\setminus \{2\}$, some properties of the space $\mathscr{M}_p$ of multipliers on $\ell^p_A$ are derived.  In particular, the failure of the weak parallelogram laws and the Pythagorean inequalities is demonstrated for $\mathscr{M}_p$.   It is also shown that the extremal multipliers on the $\ell^p_A$ spaces are exactly the monomials, in stark contrast to the $p=2$ case.  
\end{abstract}

\maketitle


\section{Introduction}

Sequence spaces play an important role in functional analysis, providing a rich source of examples, a fertile ground for generating conjectures, and a supply of applicable tools.  Indeed the theory of Banach spaces arose from early studies of the sequence space $\ell^p$.  The case $\ell^1$ is connected to the Wiener Algebra, and its additional structure has made deeper inroads possible.  The case of $\ell^2$ is particularly well understood, being isometrically isomorphic to the Hardy space $H^2$ on the open unit disk $\mathbb{D}$.  In this situation, the interplay between the analytical properties of the functions and the behavior of the space has given rise to a deep and extensive body of results, one of the great triumphs of the past century of mathematical analysis.

By contrast, when $p \neq 1$ and $p \neq 2$, relatively little is known about the space $\ell^p_A$, the space of analytic functions on  $\mathbb{D}$ for which the corresponding Maclaurin coefficients are $p$-summable.  For $1<p<\infty$, there is a notion of a $p$-inner function, in terms of which the zero sets of $\ell^p_A$ can be described \cite{Chengetal1}.  Unlike $H^2$, however, the analogous inner-outer factorization can fail when $p \neq 2$ \cite{CD}.  Whereas the multiplier algebra of $H^2$ is the familiar space $H^{\infty}$, the multipliers on $\ell^p_A$ have not been completely characterized.  

In this paper we obtain some geometric properties of the multiplier space of $\ell^p_A$.  These include the failure of the weak parallelogram laws and the Pythagorean inequalities.   In addition, we show that when $1<p<\infty$ and $p \neq 2$, the extremal multipliers on $\ell^p_A$ are exactly the monomials.  That is, the operator norm of a multiplier $\phi$ on $\ell^p_A$ coincides with its norm in $\ell^p_A$ precisely if it is of the form
\[
      \phi(z) = \gamma z^k
\]
for some $\gamma \in \mathbb{C}$ and nonnegative integer $k$.   Again, this is quite distinct from the $p=2$ case, in which the extremal multipliers  consist of the constant multiples of inner functions.


\section{The space $\ell^p_A$}

For $1\leq p \leq \infty$, the space $\ell^p_A$ is defined to be collection of analytic functions on the open unit disk $\mathbb{D}$ of the complex plane for which the Maclaurin coefficients are $p$-summable.  (The definition makes sense when $0<p<1$, but our attention will be limited to the range $1\leq p\leq \infty$.)
This function space is endowed with the norm that it inherits from the sequence space $\ell^p$.  Thus let us write 
\[
       \|f\|_p  = \|(a_k)_{k=0}^{\infty}\|_{\ell^p}
\]
for
\[
      f(z) = \sum_{k=0}^{\infty} a_k z^k
\]
belonging to $\ell^p_A$.  We stress that $\|\cdot\|_p$ refers to the norm  on $\ell^p_A$, and not the norm on $H^p$, or some other function space.

The following property is elementary, and will be essential for identifying the extremal multipliers on $\ell^p_A$ (for a proof, see  \cite[Proposition 1.5.2]{CMR}).  
\begin{Proposition}\label{elemresellp}
    If $1 \leq p_1 < p_2 \leq \infty$, then $\ell^{p_1}_A \subset \ell^{p_2}_A$, and $\|f\|_{p_2} \leq \|f\|_{p_1}$ for all $f \in \ell^{p_1}_A$.  Furthermore, $\|f\|_{p_2} = \|f\|_{p_1}$ holds  if and only if 
    \[
      f(z) = \gamma z^k
     \]
for some $\gamma \in \mathbb{C}$ and nonnegative integer $k$.
\end{Proposition}

Throughout this paper, if $1 \leq p \leq \infty$, then $p'$ will be the H\"{o}lder conjugate to $p$, that is, $1/p + 1/p' = 1$ holds.  
We recall that for $1 \leq p < \infty$, the dual space of $\ell^p_A$ can be identified with $\ell^{p'}_A$, under the pairing
\[
       \langle f, g\rangle = \sum_{k=0}^{\infty} f_k g_k,
\]
where $f(z) = \sum_{k=0}^{\infty} f_k z^k$ and $g(z) = \sum_{k=0}^{\infty} g_k z^k$ (except that complex conjugates are taken of the coefficients $g_k$ when $p=2$).  

For further exploration of $\ell^p_A$, we refer to the paper \cite{MR3714456} or the book \cite{CMR}.


\section{Orthogonality}

There is a natural way to define ``inner function'' in the context of $\ell^p_A$, that makes use of a notion of orthogonality in general normed linear spaces.  

Let $\mathbf{x}$ and $\mathbf{y}$ be vectors belonging to a normed linear space $\mathcal{X}$.  We say that $\mathbf{x}$ is {\em orthogonal} to $\mathbf{y}$ in the Birkhoff-James sense \cite{AMW, Jam}  if
\begin{equation}\label{2837eiywufh[wpofjk}
      \|  \mathbf{x} + \beta \mathbf{y} \|_{\mathcal{X}} \geq \|\mathbf{x}\|_{\mathcal{X}}
\end{equation}
for all scalars $\beta$, 
and in this case we write $\mathbf{x} \perp_{\mathcal{X}} \mathbf{y}$.

Birkhoff-James orthogonality extends the concept of orthogonality from an inner product space to normed spaces.   There are other ways to generalize orthogonality, but this approach is particularly fruitful since it is connected to an extremal condition via \eqref{2837eiywufh[wpofjk}.

 It is straightforward to check that if $\mathcal{X}$ is a Hilbert space, then the usual orthogonality relation $\mathbf{x} \perp \mathbf{y}$ is equivalent to $\mathbf{x} \perp_{\mathcal{X}} \mathbf{y}$.  More typically, however, the relation $\perp_{\mathcal{X}}$ is neither symmetric nor linear.  When $\mathcal{X} = \ell^{p}_A$,  let us write $\perp_p$ instead of  $\perp_{\ell^p_A}$.  
 
There
is an analytical criterion for the relation $\perp_p$ when $p \in (1, \infty)$.

\begin{Theorem}[James \cite{Jam}]
Suppose that $1<p<\infty$.  Then for $f(z)= \sum_{k=0}^{\infty} f_k z^k$ and $g(z)= \sum_{k=0}^{\infty} g_k z^k$ belonging to $\ell^p_A$ we have
\begin{equation}\label{BJp}
 {f} \perp_{p} {g}  \iff   \sum_{k=0}^{\infty} |f_k|^{p - 2} \overline{f}_k g_k  = 0,
\end{equation}
where any incidence of ``$|0|^{p - 2} 0$'' in the above sum  is interpreted as zero.
\end{Theorem}

In light of \eqref{BJp} we define, for a complex number $\alpha=re^{i\theta}$, and any $s > 0$, the quantity
\begin{equation}\label{definition-de-zs}
\alpha^{\langle s \rangle} = (re^{i\theta})^{\langle s \rangle} := r^{s} e^{-i\theta}.
\end{equation}
It is readily seen that for any complex numbers $\alpha$ and $\beta$,  exponent $s>0$, and integer $n \geq 0$, we have
\begin{align*}
(\alpha\beta)^{\langle s\rangle} &= \alpha^{\langle s\rangle} \beta^{\langle s\rangle}\\
|\alpha^{\langle s \rangle}| &= |\alpha|^s\\
\alpha^{\langle s \rangle}\alpha &= |\alpha|^{s+1}\\
(\alpha^{\langle s \rangle})^n &= (\alpha^n)^{\langle s \rangle}\\
(\alpha^{\langle p-1 \rangle})^{\langle p'-1 \rangle} &= \alpha.
\end{align*}

Further to the notation \eqref{definition-de-zs}, for $f(z) = \sum_{k=0}^{\infty} f_k z^k$, let us write
\begin{equation}\label{uewuweuer}
   f^{\langle s \rangle}(z)  := \sum_{k=0}^{\infty}  f_k^{\langle s \rangle} z^k
\end{equation}
for any $s>0$.  By comparing with the case  $p=2$, we can think of taking the  ${\langle s \rangle}$ power as generalizing complex conjugation.

If $f \in \ell^{p}_A$, it is easy to verify that $f^{\langle p - 1\rangle} \in \ell^{p'}$.   Thus from  \eqref{BJp} we get
\begin{equation}\label{ppsduwebxzrweq5}
f \perp_{p} g \iff \langle g, f^{\langle p - 1\rangle} \rangle = 0.
\end{equation}
Consequently the relation $\perp_{p}$ is linear in its second argument, when $p \in (1, \infty)$, and it then makes sense to speak of a vector being orthogonal to a subspace of $\ell^{p}_A$.  In particular, if $f \perp_p g$ for all $g$ belonging to a subspace $\mathscr{X}$ of $\ell^p_A$, then
\[
       \| f + g \|_p  \geq  \|f\|_p
\]
for all $g \in \mathscr{X}$.  That is, $f$ solves an extremal problem in relation to the subspace $\mathscr{M}$.

Direct calculation will also confirm that
\[
     \langle f, f^{\langle p-1\rangle} \rangle = \|f\|^p_p.
\]

With this concept of orthogonality established, we may now define what it means for a function in $\ell^p_A$ to be inner in a related sense.  
\begin{Definition}
   Let $1<p<\infty$.  A function $f \in \ell^p_A$ is said to be $p$-inner if it is not identically zero and it satisfies
   \[
          f(z)\ \, \perp_p\ \, z^k f(z)
   \]
   for all positive integers $k$.
\end{Definition}

That is, $f$ is nontrivially orthogonal to all of its forward shifts.  Apart from a harmless multiplicative constant, this definition is equivalent to the traditional meaning of ``inner'' when $p=2$.  Furthermore, this approach to defining an inner property is consistent with that taken in other function spaces \cite{MR1440934, Dima, CMR2,CMR3,  MR1278431,MR1398090,MR1197044,MR936999,Seco}.

Birkhoff-James Orthogonality also plays a role when we examine a version of the Pythagorean theorem for normed spaces in Section \ref{geommp}.


\section{Multipliers on $\ell^p_A$}

An analytic function $\phi$ on $\D$ is said to be a {\em multiplier} of $\ell_{A}^{p}$ if
$$f \in \ell^{p}_{A} \implies \phi f \in \ell^{p}_{A}.$$
The set of multipliers of $\ell^{p}_{A}$ will be denoted by $\M_{p}$.  

For $\phi \in \M_{p}$, an application of the closed graph theorem shows that the linear mapping
$$M_{\phi}: \ell^{p}_{A} \to \ell^{p}_{A}, \quad M_{\phi} f = \phi f$$ is continuous. Thus we can define the {\em multiplier norm} of $\phi$ by
$$\|\phi\|_{\M_p} := \sup\{\|\phi f\|_{p}: f \in \ell^{p}_{A}, \|f\|_{p} \leq 1\}.$$
In other words, the multiplier norm of $\phi$ coincides with the operator norm of $M_{\phi}$ on $\ell^{p}_{A}$.  Henceforth we identify the multiplication operator $M_{\phi}$ with its symbol $\phi$.

Relatively little is known about the multipliers on $\ell^p_A$, except when $p= 1$ or $p= 2$.  In the former case, we know that $\mathscr{M}_1 = \ell^1_A$, and in the latter, $\mathscr{M}_2 = H^{\infty}$.   We will accordingly concentrate our efforts on the range $1<p<\infty$, with $p\neq 2$.

The following basic results have been established in the literature.

\begin{Proposition}  Let $1<p<\infty$.
   If $\phi \in \mathscr{M}_p$, then $\phi \in H^{\infty} \cap \ell^p_A \cap \ell^{p'}_A$, and $\mathscr{M}_p = \mathscr{M}_{p'}$, with $\|\phi\|_{\mathscr{M}_{p}} = \|\phi\|_{\mathscr{M}_{p'}}$.
\end{Proposition}

\begin{Proposition}\label{elempropmult}   Let $1<p<\infty$.
   If $\phi(z) = \sum_{k=0}^{\infty} \phi_k z^k  \in \mathscr{M}_p$, then $\|\phi\|_p \leq \|\phi\|_{\mathscr{M}_p} \leq \|\phi\|_1$ (with $\|\phi\|_1 = \infty$ being possible), and
   \[
        |\phi_0| + |\phi_1| + \cdots + |\phi_n| \leq \|\phi\|_{\mathscr{M}_p} (n+1)^{1/p'}.
   \]
   If all of the coefficients of $\phi$ are nonnegative, then $\phi \in \ell^1_A$, and $\|\phi\|_1 = \|\phi\|_{\mathscr{M}_p}$.
\end{Proposition}

Define the difference quotient mapping $Q_w$ by
\[
      Q_w f(z) := \frac{f(z)-f(w)}{z-w}
\]
for any $w \in \mathbb{D}$ and analytic function $f$ on $\mathbb{D}$. 

Difference quotients are (bounded) operators on $\mathscr{M}_p$.  In fact, for any multiplier $\phi$ on $\ell^p_A$, and $w \in \mathbb{D}$,
\[
        \|Q_w \phi\|_{\mathscr{M}_p}  \leq  \frac{1}{1 - |w|} (\|\phi\|_{\mathscr{M}_p} + \phi(w)).
\]   
 
 For proofs of these multiplier properties, see \cite[Chapter 12]{CMR}, which has references to original sources.

To extract some geometric information about $\mathscr{M}_p$, we will rely on the following observation.
\begin{Corollary}\label{linmultell1}
     For any complex numbers $\alpha$ and $\beta$, the multiplier $\phi(z) = \alpha + \beta z$ satisfies
     \[
            \|\phi\|_{\mathscr{M}_p}  =  \|\phi\|_1 = |\alpha|+|\beta|.
     \]
\end{Corollary}

\begin{proof}
The claim is trivial if $\alpha=0$ or $\beta = 0$.  Otherwise,
the mapping
\[
     f(z) \longmapsto  f\Big( \frac{\alpha \bar{\beta}}{| \alpha \bar{\beta}|} z\Big)
\]
determines a linear isometry on $\ell^p_A$ (in fact it is unitary).

Consequently, the multiplier $\phi$
 has the same norm as the multiplier
 \[
   \phi\Big( \frac{\alpha \bar{\beta}}{| \alpha \bar{\beta}|} z\Big) =  \frac{1}{\bar{\alpha}} \Bigg(  |\alpha|^2 + \beta \bar{\alpha} \Big( \frac{\alpha \bar{\beta}}{| \alpha \bar{\beta}|} z\Big)   \Bigg) =  \frac{1}{\bar{\alpha}} \big( |\alpha|^2 + |\beta \bar{\alpha}| z \big),
 \]  
 which is $|\alpha| + |\beta|$, according to the last part of Proposition \ref{elempropmult}.
 \end{proof}

 Already this delivers some information about the geometry of $\mathscr{M}_p$.
 \begin{Corollary}
If $1<p<\infty$, then $\mathscr{M}_p$ fails to be strictly convex.  
 \end{Corollary}
 \begin{proof}
     The unit ball in $\mathscr{M}_p$ contains the segment $t + (1-t)z$ for $0\leq t\leq 1$. 
 \end{proof}

It is known that certain Blaschke products are multipliers of $\ell^p_A$ (e.g., if the zeros converge to the boundary rapidly enough), and that certain other classes of functions are multipliers. However, there does not yet exist a complete characterization of $\mathscr{M}_p$ in terms of the coefficients, or of the boundary function.  Our sources on the subject  include \cite{MR0174937,MR0121655,MR1296567,MR1678853,MR1771763,MR2238173,MR1373751,MR0193513,MR583804,MR0352958,MR592495}, along with the survey paper \cite{MR3714456}.


\section{The geometry of $\mathscr{M}_p$}\label{geommp}

It is well known that when $1<p<\infty$, the spaces $\ell^p$ (and hence also $\ell^p_A$) are uniformly convex and uniformly smooth.  In fact, more can be said.  A normed space $\mathscr{X}$ is said to satisfy the Lower Weak Parallelogram property (LWP) with constant $C>0$ and exponent $r>1$, if
\[
      \| x + y \|^r_{\mathscr{X}} + C\| x - y \|^r_{\mathscr{X}} \leq 2^{r-1}(\|x\|_{\mathscr{X}}^r + \|y\|_{\mathscr{X}}^r)
\]
for all $x$ and $y$ in $\mathscr{X}$;  it satisfies the Upper Weak Parallelogram property (UWP) if for some (possibly different) constant and exponent the reverse inequality holds for all $x$ and $y$ in $\mathscr{X}$.    If $\mathscr{X}$ is a Hilbert space, then the parallelogram law holds, corresponding to $r=2$ and $C=1$.   Otherwise, these inequalities generalize Clarkson's inequalities \cite{Clark}, and the parameters $r$ and $C$ give a sense of how far the space $\mathscr{X}$ departs from behaving like a Hilbert space.

It was shown in \cite{CR} that the $L^p$ spaces satisfy LWP and UWP when $1<p<\infty$, and the full ranges of parameters $C$ and $r$ were identified (see also \cite{Byn,BD,CH2,CMP1,CRM2}).  More generally, a space satisfying LWP is uniformly convex, and a space satisfying UWP is uniformly smooth \cite[Proposition 3.1]{CR}.  From this it could be further surmised that the dual of a LWP space is an UWP space, and vice-versa; this is made precise in \cite[Theorem 3.1]{CH2}.

Another useful consequence of the weak parallelogram laws is a version of the Pythagorean Theorem for normed spaces, where orthogonality is in the  Birkhoff-James sense.  It takes the form of a family of inequalities relating the lengths of orthogonal vectors with that of their sum \cite[Theorem 3.3]{CR}.

\begin{Theorem}[\cite{CR}]
   If  a smooth Banach space $\mathscr{X}$ satisfies LWP with constant $C>0$ and exponent $r>1$,  then there exists $K>0$ such that 
\begin{equation}\label{LPyth}
  \| x \|^r_{\mathscr{X}}  +  K\|y\|^r_{\mathscr{X}} \leq \|x + y \|^r_{\mathscr{X}}
\end{equation}
whenever  $x \perp_{\mathscr{X}}  y$;
if $\mathscr{X}$ satisfies UWP with constant $C>0$ and exponent $r>1$,  then there exists a positive constant $K$ such that  
\begin{equation}\label{UPyth}
   \| x \|^r_{\mathscr{X}}  +  K\|y\|^r_{\mathscr{X}} \geq \|x + y \|^r_{\mathscr{X}}
\end{equation}
whenever  $x \perp_{\mathscr{X}}  y$.
In either case, the constant $K$ can be chosen to be ${C}/{(2^{r-1}-1)}$
\end{Theorem}

When $\mathscr{X}$ is any Hilbert space, the parameters are $K=1$ and $r=2$, and the Pythagorean inequalities reduce to the familiar Pythagorean theorem.  More generally, these Pythagorean inequalities enable the application of some Hilbert space methods and techniques to smooth Banach spaces satisfying LWP or UWP; see, for example, \cite[Proposition 4.8.1 and Proposition 4.8.3; Theorem 8.8.1]{CMR}.

The weak parallelogram laws and the Pythagorean inequalities fail on $L^1$ and $L^{\infty}$.
We previously saw in Corollary \ref{linmultell1} that $\mathscr{M}_p$ contains a subspace, consisting of the linear functions, that behaves geometrically like $\ell^1_A$.  Consequently we would expect the weak parallelogram laws and the Pythagorean inequalities to fail on $\mathscr{M}_p$, and indeed that is the case.

\begin{Theorem}
     Let $1<p<\infty$.  The space $\mathscr{M}_p$ fails to satisfy LWP or UWP for any constant $C>0$ or exponent $r>1$.
\end{Theorem}

\begin{proof}
If 
\[
   \|1+z\|^r_{\mathscr{M}_p} + C\|1-z\|^r_{\mathscr{M}_p} \leq 2^{r-1}(\|1\|_{\mathscr{M}_p}^r + \|z\|^r_{\mathscr{M}_p} )
\]
holds, then
\[
     (1 + C)2^r \leq 2^{r-1}(2),
\]
which forces $C=0$.  Thus LWP fails.

Similarly, for $C>0$ we have
\[
     \|1\|^r_{\mathscr{M}_p} + C^r\|z/C\|^r_{\mathscr{M}_p}  \geq 2^{r-1}(\|1+z/C\|^r_{\mathscr{M}_p}+\|1-z/C\|^r_{\mathscr{M}_p})
\]
implies
\[
    2  \geq 2^{r-1}\cdot 2\cdot(1 +1/C)^r,
\]
which is absurd when $1<r<\infty$.  Therefore UWP also fails.
\end{proof}

\begin{Theorem}
    Let $1<p<\infty$.  The space $\mathscr{M}_p$ fails to satisfy either of the Pythagorean inequalities for any parameters $r>1$ and $K>0$.
\end{Theorem}

\begin{proof}
Fix $c\neq 0$.  Let $\phi(z) = 1 + cz$, and consider $f(z) \in \ell^p_A$ of the form $f(z) = f_0 + f_2 z^2 + f_4 z^4 + \cdots$.  Then
\begin{align*}
   \|  \phi(z)f(z)\|_p^p &=   \|(1 + cz)(f_0 + f_2 z^2 + f_4 z^4 + \cdots)\|^p_p \\
        &=  \|(f_0 + f_2 z^2 + f_4 z^4 + \cdots) +  cz(f_0 + f_2 z^2 + f_4 z^4 + \cdots)\|_p^p \\
        &=   |f_0|^p + |c|^p|f_0|^p + |f_1|^p + |c|^p|f_1|^p + |f_2|^p + |c|^p|f_2|^p \\
        &=   \|f\|_p^p  +  |c|^p \|f\|_p^p\\
        &\geq \|f\|_p^p.
\end{align*}

This shows that $\| 1 + cz\|_{\mathscr{M}_p} \geq \|1\|_{\mathscr{M}_p}$ for all constants $c$, or $1  \perp_{\mathscr{M}_p}  z$.  By considering the limit
\[
      \lim_{c \rightarrow 0} \frac{\|1+cz\|_{\mathscr{M}_p}^r - \|1\|_{\mathscr{M}_p}^r}{\|cz\|_{\mathscr{M}_p}^r} = \lim_{c\rightarrow 0} \frac{(1 + |c|)^{r} - 1^r}{|c|^r},
\]
we see that $\mathscr{M}_p$ fails to satisfy \eqref{UPyth}, as $K=\infty$ would be forced.

Next, note  that for $c\neq 0$, we have
\begin{align*}
     \|(1+z) + c(1-z)\|_{\mathscr{M}_p}  &=  \|(1 + c) + (1-c)z\|_{\mathscr{M}_p}   \\
             &=  |1+c| + |1-c| \\
             &\geq 2\\
             &=  \|1+z\|_1 \\
             &= \|1+z\|_{\mathscr{M}_p}.
\end{align*}

This shows that $1+z \perp_{\mathscr{M}_p} 1-z$.  Next, consider
\[
      \frac{\|(1+z) + c(1-z)\|_{\mathscr{M}_p}^r - \|1+z\|_{\mathscr{M}_p}^r}{\|c(1-z)\|_{\mathscr{M}_p}^r}= \frac{(|1+c|+|1-c|)^r  -  2^r}{2|c|^r},
\]
where $1<r<\infty$.

This tends toward zero as $c \rightarrow 0^+$, which would require $K=0$.  Thus, \eqref{LPyth} fails to hold.
\end{proof}


\section{Functionals on $\mathscr{M}_p$}

Let $1<p<\infty$.  Suppose that $\lambda = (\lambda_0, \lambda_1, \lambda_2,\ldots)$ is a sequence of complex numbers such that for some $C>0$ we have
\[
      | \lambda_0 \phi_0 + \lambda_1 \phi_1 + \lambda_2 \phi_2 +\cdots| \leq C \|\phi\|_{\mathscr{M}_p}
\]
for all $\phi(z) = \sum_{k=0}^{\infty} \phi_k z^k \in \mathscr{M}_p$.   Then $\lambda$ determines a bounded linear functional on $\mathscr{M}_p$ with norm at most $C$.  Let us give the name $\mathscr{S}=\mathscr{S}_p$ to the collection of functionals arising in this manner.  It is a nonempty collection, since it contains all of $\ell^{p'}_A$.   Thus $\mathscr{S}$ is a linear manifold within $\mathscr{M}_p^*$, the continuous dual space of $\mathscr{M}_p$.

If 
$\lambda = (\lambda_0, \lambda_1, \lambda_2,\ldots)\in \mathscr{S}$, then $\lambda_k = \lambda(z^k)$ and the pairing
\[
       \lambda(\phi) = \sum_{k=0}^{\infty}  \lambda_k \phi_k.
\]
applies for all $\phi \in \mathscr{M}_p$.

Trivially, we can bound the norm of $\lambda$ as follows:
\begin{equation}\label{mpstarnormineq}
     \|\lambda\|_{p'}  =  \sup_{\phi\neq 0}\frac{|\lambda(\phi)|}{\|\phi\|_p} \geq \sup_{\phi\neq 0}\frac{|\lambda(\phi)|}{\|\phi\|_{\mathscr{M}_p}} = \|\lambda\|_{\mathscr{M}_p^*}
     \geq \sup_{\phi\neq 0}\frac{|\lambda(\phi)|}{\|\phi\|_1} = \|\lambda\|_{\infty},
\end{equation}
possibly with $\infty$ on the left side. 

Of course since $\mathscr{M}_p \subseteq \ell^p_A \cap \ell^{p'}_A$ we also have
\[
     \|\lambda\|_p  \geq \|\lambda\|_{\mathscr{M}_p^*},
\]
again, with the left side possibly being infinite.

Taking the $\langle p-1 \rangle$ power does something natural in this context.
\begin{Proposition} Let $1<p<\infty$. 
    If $\phi \in \mathscr{M}_p$, then $\phi^{\langle p-1 \rangle} \in \mathscr{S}$.
\end{Proposition}

\begin{proof}
    In this situation, $\phi^{\langle p-1 \rangle}  \in \ell^{p'}_A$, and hence $\phi^{\langle p-1 \rangle} \in \mathscr{S}$, by \eqref{mpstarnormineq}.
\end{proof}


Members of $\mathscr{S}$ might not have radial boundary limits, but they do satisfy the following growth condition, which can also be interpreted as boundedness of point evaluation.
\begin{Proposition} Let $1<p<\infty$. 
    If $\lambda \in \mathscr{S}$, then 
    \[
        |\lambda(w)| \leq  \frac{\|\lambda\|_{\mathscr{M}^*_p}}{(1 - |w|^p)^{1/p}}, \ \, w \in \mathbb{D}.
\] 
\end{Proposition}

\begin{proof}
The point evaluation functional at any point $w$ of the disk is a multiplier on $\ell^p_A$, since its coefficients are absolutely summable.  Apply $\lambda \in \mathscr{S}$ to point evaluation at $w$ to see that $\lambda$ satisfies the growth condition shown above.
\end{proof}

It turns out that difference quotients are bounded on $\mathscr{S}$.
\begin{Proposition}
    Let $1<p<\infty$.
    If $\lambda \in \mathscr{S}$, and $w \in \mathbb{D}$, then $Q_w \lambda \in \mathscr{S}$, and
    \[
          \|Q_w \lambda \|_{\mathscr{M}_p^*} \leq \frac{\|\lambda\|_{\mathscr{M}_p^*}}{1 - |w|}.
    \]
\end{Proposition}

\begin{proof}
It is easy to see that if $\phi \in \mathscr{M}_p$, then $S^k \phi$ belongs to $\mathscr{M}_p$ for all $k\geq 0$, with equal norms.
We now calculate
\begin{align*}
      (Q_w \lambda)(\phi)  &=  \Bigg(\frac{\sum \lambda_k z^k - \sum \lambda_k w^k}{z-w}\Bigg)(\phi)  \\
           &=  \Big(\sum_{k=1}^{\infty} \lambda_k(z^{k-1} + z^{k-2}w + \cdots + w^{k-1}) \Big)(\phi)\\
           &= \Big( \sum_{k=1}^{\infty} \sum_{j=0}^{k-1} \lambda_k z^j w^{k-j-1}\Big)(\phi)\\
           &=  \sum_{k=1}^{\infty} \sum_{j=0}^{k-1} \lambda_k \phi_j w^{k-j-1} \\
           &=  \lambda_1(\phi_0)\\
           &\quad\, + \lambda_2(\phi_0 w + \phi_1)\\
           &\quad\, + \lambda_3(\phi_0 w^2 + \phi_1 w + \phi_2)\\
           &\quad\, + \cdots\\
           &= \lambda(S\phi) + w \lambda(S^2\phi) + w^2 \lambda(S^3\phi) + \cdots.
\end{align*}

From this we obtain
\begin{align*}
     |(Q_w \lambda)(\phi)|  &\leq   \|\lambda\|_{\mathscr{M}_p^*} \|S \phi\|_{\mathscr{M}_p} +  |w|\|\lambda\|_{\mathscr{M}_p^*} \|S^2 \phi\|_{\mathscr{M}_p}
     +|w|^2\|\lambda\|_{\mathscr{M}_p^*} \|S^3 \phi\|_{\mathscr{M}_p}+\cdots \\
     &= \frac{\|\lambda\|_{\mathscr{M}_p^*}\|\phi\|_{\mathscr{M}_p}}{1 - |w|},
\end{align*}
which proves the claim.
\end{proof}

Let us add that the weak parallelogram laws and the Pythagorean inequalities must fail on $\mathscr{M}_p^*$ as well.  This is because it contains a subspace that is isomorphic to $\ell^{\infty}_A(\{0,1\})$.  
Furthermore we see that $\mathscr{M}_p$ fails to be smooth.  For example, the multiplier $1$ is normed by both $1$ and $1+z$ in $\mathscr{M}_p^*$.


\section{The extremal multipliers on $\ell^p_A$}

Recall that if $\phi \in \mathscr{M}_p$, then $\|\phi\|_{\mathscr{M}_p} \geq \|\phi\|_p$.   We say that the multiplier $\phi$ is $\it extremal$ if equality holds, that is,
\[
     \|\phi\|_{\mathscr{M}_p} = \|\phi\|_p.
\]

For $\ell^2_A = H^2$, the multipliers are the bounded analytic functions on $\mathbb{D}$, and the extremal multipliers are exactly the constant multiples of inner functions.  Indeed, if 
\[
     \sup_{z \in \mathbb{D}} |\phi(z)| = \sup_{0<r<1} \int_{\mathbb{T}} |\phi(re^{i\theta})|^2\,\frac{d\theta}{2\pi},
\]
then $|\phi(e^{i\theta})| = \|\phi\|_{H^{\infty}}$ a.e.\ is forced. The reverse implication is similarly trivial.

For $\ell^p_A$, $p\neq 2$, it would therefore be plausible to guess that the extremal multipliers are the $p$-inner functions.  However, this is incorrect.
It was shown in \cite{Chengetal1} that for $2<p<\infty$, there are $p$-inner functions whose zero sets fail to be Blaschke sequences.  Such a $p$-inner function cannot be a multiplier of $\ell^p_A$, since it would also have to belong to $\ell^{p'}_A$.
In the paper \cite{Che}  $p$-inner functions are constructed whose zero sets accumulate at every point of the boundary circle $\mathbb{T}$.    However, by \cite[Corollary 12.6.3]{CMR},  a multiplier on $\ell^p_A$  for $p \in [1,2)$ has unrestricted limits almost everywhere on $\mathbb{T}$.  A $p$-inner function thus described cannot therefore be a multiplier on $\ell^p_A$.  

More can be said when $p \neq 2$.  First, the extremality of a multiplier is inherited by its conjugate in the following sense.
\begin{Proposition}
   Let $\phi \in \mathscr{M}_p$.  If $\|\phi\|_{\mathscr{M}_p} = \|\phi\|_p$, then $\|\lambda\|_{\mathscr{M}_p^*} = \|\lambda\|_{p'}$, where $\lambda = \phi^{\langle p-1 \rangle}$.
\end{Proposition}

\begin{proof}
     By hypothesis,
     \[
          \|\phi\|_{\mathscr{M}_p} = \sup \frac{|\langle \phi f, g\rangle|}{\|f\|_p \|g\|_{p'}} = \sup \frac{|\langle \phi 1, g\rangle|}{\|1\|_p \|g\|_{p'}} = \|\phi\|_p.
     \]
     This forces $g = \phi^{\langle p-1 \rangle}$, apart from a harmless multiplicative constant.
     
     Since $g \in \ell^{p'}$, we also have $g \in \mathscr{S}$ by \eqref{mpstarnormineq}. Relabeling $g$ as the functional $\lambda$, we have
     \[
         \|\phi\|_p =  \frac{|\langle \phi, g\rangle|}{\|g\|_{p'}}  \leq  \frac{|\langle \phi, \lambda \rangle|}{\|\lambda\|_{\mathscr{M}_p^*}} \leq \|\phi\|_{\mathscr{M}_p}.
     \]
      Equality is forced throughout, and we conclude that 
      \[
            \|\lambda\|_{p'} = \|\lambda\|_{\mathscr{M}_p^*}.
      \]   
\end{proof}
This comes into play in the main result, to which we presently turn.

\begin{Theorem}\label{ExtremMultMonom}
    Let $p\in (1,\infty)\setminus\{2\}$.  A multiplier $\phi \in \mathscr{M}_p$ satisfies $\|\phi\|_{\mathscr{M}_p} = \|\phi\|_p$ if and only if $\phi$ is a monomial.
\end{Theorem}

\begin{proof}

%
First, the claim is trivial if $\phi$ is identically zero, so let us suppose otherwise.  Also,
since $\mathscr{M}_p = \mathscr{M}_{p'}$ as point sets and with equal norms, it follows $\mathscr{M}_p^* = \mathscr{M}_{p'}^*$ with equal norms as well.

Now suppose that $2<p<\infty$.  Then $1<p'<2$, and we have
   \[
        \|\phi\|_{p'} \geq \|\phi\|_p  = \|\phi\|_{\mathscr{M}_p} = \|\phi\|_{\mathscr{M}_{p'
        }}  \geq \|\phi\|_{p'}.
   \]
   Equality is forced throughout.  In particular, $\|\phi\|_p = \|\phi\|_{p'}$, which implies that $\phi$ is a monomial, according to Proposition \ref{elemresellp}.  (This step fails if $p = p'= 2$).

Finally, let $1<p<2$, and suppose that $\phi \in \mathscr{M}_p$ is extremal; that is, $\|\phi\|_{\mathscr{M}_p} = \|\phi\|_p$.  Then
\begin{align*}
    \|\phi\|_{\mathscr{M}_p}  &=  \|\phi\|_p  \\
      &\geq \|\phi\|_{p'} \\
      &\geq \|\phi\|_{\mathscr{M}_{p'}^*} \\
      &\geq \frac{|\langle \phi^{\langle p-1\rangle}, \phi \rangle|}{\|\phi^{\langle p-1\rangle}\|_{\mathscr{M}_{p'}}} \\
      &= \frac{\|\phi\|^p_p}{\|\phi^{\langle p-1\rangle}\|_{\mathscr{M}_{p'}}}  \ \ \ (\star) \\
      &= \frac{\|\phi\|^p_p}{\|\phi^{\langle p-1\rangle}\|_{p'}}\\
      &= \|\phi\|_p. 
\end{align*}
This forces $\phi$ to be a monomial.

From the line ($\star$) to the next, we used $\|\phi^{\langle p-1\rangle}\|_{\mathscr{M}_{p'}} = \|\phi^{\langle p-1\rangle}\|_{p'}$, which we derive as follows:
\begin{align*}
   \|\phi\|_p^{p-1} &= \frac{|\langle \phi^{\langle p-1\rangle}, \phi \rangle|}{\|\phi\|_{p}}\\ 
      &=  \frac{|\langle \phi^{\langle p-1\rangle}, \phi \rangle|}{\|\phi\|_{\mathscr{M}_p}} \\
      &\leq  \|\phi^{\langle p-1\rangle}\|_{\mathscr{M}_p^*} \\
      &=  \|\phi^{\langle p-1\rangle}\|_{\mathscr{M}_{p'}^*} \\
      &\leq \|\phi^{\langle p-1\rangle}\|_{p'} \\
      &=  \|\phi\|_p^{p-1}
\end{align*}
and equality must hold throughout.

Conversely, if any monomial multiplier, and it can be checked by inspection that it is extremal.
\end{proof}

\begin{Example}
If $1<p<\infty$ and $0 < |w|<1$, then the function
\[
     B(z)  :=  \frac{1 - z/w}{1 - w^{\langle p'-1 \rangle}z},
\] 
turns out to be $p$-inner \cite[Lemma 3.2]{MR3686895}, and 
\[
     \|B\|_p^p  =  1 + \frac{(1 - |w|^{p'})^{p-1}}{|w|^p}.
\]
Note, in particular, that when $p=2$ the function $B$ is the Blaschke factor, possibly apart from a multiplicative constant, with its root at $w$.

Since $B$ is analytic in a neighborhood of the closed disk $\overline{\mathbb{D}}$, it is a multiplier.  Let us show directly that for $p=4$ it fails to be extremal, as must be the case according to Theorem \ref{ExtremMultMonom}.

We will take as a test function
\[ 
     f(z) := 1 - w^{\langle {p'}-1 \rangle}z,
\]
so that $f \in \ell^p_A$ and 
\[
    \|f\|_p^p =  1 + 1/|w|^p.
\]

Now fix $p=4$, so that $p' = 4/3$.
For $0<a<1$ we have  the elementary inequalities
\begin{align*}
      a - a^2 &> 0\\
      3(a^2 - a) &< a^2 - a \\
      a^3 + 1 - 3a + 3a^2 - a^3 &< a^2 - a + 1 \\
      a^3 + (1 - a)^3  &<  \frac{a^3 + 1}{1+a}\\
      1 + \frac{(1 - a)^3}{a^3} &<  \frac{1 + 1/a^3}{1 + a}.\\
\end{align*}

Substitute $a = |w|^{4/3}$ to obtain
\[
      1 + \frac{(1 - |w|^{4/3})^{4-1}}{|w|^{4}} <  \frac{1 + 1/|w|^{4}}{1 + |w|^{4/3}}. 
\]

This yields the bound
\begin{align*}
      \|B\|_p^p  
           &= 1 + \frac{(1 - |w|^{4/3})^{p-1}}{|w|^{4}}\\
           &< \frac{1 + 1/|w|^{4}}{1 + |w|^{4/3}}\\
           &= \frac{\|Bf\|_p^p}{\|f\|_p^p}\\
           &\leq \|B\|_{\mathscr{M}_p}^p.
\end{align*}

This verifies that $B$ fails to be an extremal multiplier.
\end{Example}


\bibliographystyle{plain}

\bibliography{referencesGeomMultEllp}
   
\end{document}